\documentclass[11pt,a4paper]{article}
\usepackage[T1]{fontenc}
\usepackage{lmodern}
\usepackage{amsmath,amssymb,amsthm,mathtools}

\usepackage[margin=29mm,headheight=14pt]{geometry}
\usepackage{microtype}
\usepackage{enumitem}
\usepackage{fancyhdr}
\usepackage{xcolor}
\definecolor{linkblue}{rgb}{0,0,0.45}
\usepackage[colorlinks=true,linkcolor=linkblue,citecolor=linkblue,
            urlcolor=linkblue,unicode]{hyperref}

\hypersetup{pdftitle={Non-endpoint Fourier restriction for flat perturbations of curved-root surfaces},pdfsubject={Fourier restriction, flat perturbations, finite-type surfaces},pdfauthor={}}
\numberwithin{equation}{section}
\newtheorem{theorem}{Theorem}[section]
\newtheorem{corollary}{Corollary}[section]
\newtheorem{proposition}[theorem]{Proposition}
\newtheorem{lemma}[theorem]{Lemma}
\theoremstyle{remark}
\newtheorem{remark}[theorem]{Remark}
\newcommand{\R}{\mathbb R}
\newcommand{\E}{\mathcal E}
\newcommand{\dd}{\,\mathrm d}
\newcommand{\supp}{\operatorname{supp}}

\newcommand{\norm}[1]{\lVert #1\rVert}
\newcommand{\abs}[1]{\lvert #1\rvert}
\newcommand{\qc}{q_{\mathrm c}}
\newcommand{\pc}{p_{\mathrm c}}
\newcommand{\hlin}{h_{\mathrm{lin}}}
\newcommand{\hrefdoi}[1]{\href{https://doi.org/#1}{\nolinkurl{#1}}}

\let\oldthanks\thanks
\renewcommand{\thanks}[1]{%
	\renewcommand{\thefootnote}{}
	\oldthanks{#1}
	\renewcommand{\thefootnote}{\arabic{footnote}}
}

\date{}

\begin{document}
\title {\textbf{Stability of Fourier Restriction for a Class of Finite Type Surfaces in $\mathbb R^3$}	
	\author{Wenjuan Li, \hspace{0.2cm}Hongqiang Wang, \hspace{0.2cm}Huiju Wang}
	\hspace{0.2cm}
	\thanks{This work is supported by the National Key R\&D Program of China (No.2023YFA1010800); Natural Science Foundation of China (No.12301113; No.12271435).}}

\date{}
\maketitle
\thispagestyle{empty}

\begin{abstract}
We establish the stability of the optimal non-endpoint Fourier restriction range under smooth flat perturbations of curved-root surfaces in $\mathbb R^3$. Specifically, we consider graphs of
$$\phi(x_1,x_2)=(x_2-x_1^m)^n+R(x_1,x_2),$$
where $m,n\ge2$ are integers and $R$ is smooth and flat at the origin. On a sufficiently small neighborhood depending on the perturbation, we prove the local $L^p(\mathbb R^3)\to L^2(S,d\sigma)$ restriction estimate for $1\le p<p_c$, where
$$p_c=\frac{q_c}{q_c-1}, \qquad q_c=\max\left\{2+\frac{2mn}{m+1},2n\right\}.$$
If the nonnegative localizing density is positive at the origin, an adapted-box example shows that the estimate fails for $p>p_c$. Thus arbitrary smooth flat perturbations preserve the optimal restriction range away from the critical endpoint. For $n\ge3$, the main local estimate combines planar curve decoupling with a one-dimensional finite-type extension estimate, whose principal scale factors cancel exactly. For $n=2$, we prove a uniform $C^4$ perturbation estimate using a weighted convolution bound and an anisotropic Whitney decomposition. Anisotropic rescaling and summation complete the proof. 

\end{abstract}

\section{Introduction}

\subsection{Background}

Let $D \subset \R^2$ be a bounded open neighborhood of the origin,
and let $\phi \in C^\infty(D;\R)$ satisfy $\phi(0)=0$ and $\nabla\phi(0)=0$.
Writing
\[
 \Gamma_\phi(\xi)=(\xi_1,\xi_2,\phi(\xi)),
\]
we define the Fourier extension operator
\begin{equation}\label{eq:extension}
 \E_{\phi,D}g(x)
 =
 \int_D e^{ix\cdot\Gamma_\phi(\xi)}g(\xi)\dd \xi,
 \qquad x\in\R^3,
\end{equation}
for $g\in L^2(D)$.
The local $L^2$ restriction problem asks for the exponents $q$
for which
$\E_{\phi,D}:L^2(D)\to L^q(\R^3)$
is bounded on a sufficiently small neighborhood $D$ of the origin.
By duality, this is equivalent to the restriction estimate
\[
 \|\widehat f\circ\Gamma_\phi\|_{L^2(D)}
 \lesssim
 \|f\|_{L^{q'}(\R^3)}
\]
for Schwartz functions $f$, where $q'$ is the H\"older conjugate of $q$.

The Newton-polyhedral theory of Ikromov and M\"uller relates the critical $L^2$ restriction exponent for finite-type surfaces in $\R^3$ 
to the restriction height of the defining phase; see \cite{IMbook,IMI}.
Their sufficiency theorem for smooth phases in nonadapted coordinates assumes a factorization condition, denoted by Condition~\textup{(R)} \cite[Theorem~1.7]{IMI}.
A flat perturbation of the phase function leaves the Taylor series, and hence the Newton polyhedron in the given coordinates, unchanged,
but need not preserve this smooth factorization.
The resulting question is whether the sharp restriction range remains stable when the Taylor data are fixed but the factorization is lost.

We address this question for the family
\begin{equation}\label{eq:phase}
 \phi(\xi_1,\xi_2) = (\xi_2-\xi_1^m)^n+R(\xi_1,\xi_2),
 \qquad m,n\geq2,
\end{equation}
where $m,n$ are integers and $R$ is smooth and flat at the origin:
\begin{equation}\label{eq:flat}
 \partial^\alpha R(0)=0
 \qquad\text{for every multi-index }\alpha.
\end{equation}
The stability problem for this family was posed by M\"uller
in \cite[Section~6, Problem~8]{Problems}.
We prove that, for each such $R$, the full non-endpoint
$L^2$ extension range of the unperturbed model remains valid
on a sufficiently small neighborhood of the origin.

Related questions have been studied for other classes of
surfaces and measures.
Buschenhenke, M\"uller, and Vargas considered flat
one-variable perturbations of the hyperbolic paraboloid
\cite{BMV}, while Li established smooth affine restriction
estimates with curvature-weighted measures \cite{Li}.
Here we consider unweighted restriction for perturbations
of a model with the curved multiple root $\xi_2=\xi_1^m$.

Our proof combines Palle's restriction estimate away from
the principal root \cite[Proposition~4.3]{Palle}
with uniform local estimates near that root.
These local estimates provide the stability needed to
handle the general two-variable remainder in \eqref{eq:phase}.

\subsection{Main results}

Define
\begin{equation}\label{eq:critical}
 \qc:= \max\left\{2+\frac{2mn}{m+1},\,2n\right\},
 \qquad \pc:=\frac{\qc}{\qc-1}.
\end{equation}
For the unperturbed model
$\phi_0(\xi)=(\xi_2-\xi_1^m)^n$, the critical extension exponent
is $\qc$. More precisely, for every sufficiently small bounded
open neighborhood $U$ of the origin,
\[
 \norm{\E_{\phi_0,U}}_{L^2(U)\to L^q(\R^3)}<\infty
 \iff
 q\in[\qc,\infty].
\]
The critical exponent is given in
\cite{IMI}.

We extend the non-endpoint range to the case that $\phi_0$ is perturbed by a smooth function which is flat at the origin.

\begin{theorem}\label{thm:main}
Suppose that $\phi$ is given by \eqref{eq:phase}, where $R$ is real-valued, smooth, and flat at the origin.
There is a bounded open neighborhood $U$ of the origin, depending on $m,n$ and $R$, 
such that, for every $q\in(\qc,\infty]$,
\begin{equation}\label{eq:mainextension}
 \norm{\E_{\phi,U}g}_{L^q(\R^3)} \leq C_q\norm{g}_{L^2(U)},
 \qquad g\in L^2(U).
\end{equation}
Consequently, for $S=\Gamma_\phi(U)$ and every nonnegative $\eta\in C_c^\infty(S)$, the restriction estimate
\begin{equation}\label{eq:mainrestriction}
 \left(\int_S\abs{\widehat f}^{\,2}\dd\mu\right)^{1/2}
 \leq C_{p,\eta}\norm{f}_{L^p(\R^3)},
 \qquad 1\leq p<\pc,
\end{equation}
holds for all $f\in\mathcal{S}(\R^3)$, where $\dd\mu=\eta\dd\sigma$ and $\dd\sigma$ denotes surface measure.
Moreover, if $\eta(\Gamma_\phi(0))>0$, then \eqref{eq:mainrestriction} can hold only if $p\leq\pc$.
\end{theorem}

To relate $\qc$ to the Newton geometry, let
$d(\phi)$, $\hlin(\phi)$, $h(\phi)$, and $h^r(\phi)$
denote the Newton distance, linear height, height, and
restriction height, respectively, with the conventions
of \cite{IMbook,IMI}.
Since a flat term remains flat under smooth coordinate
changes fixing the origin, these quantities agree with
those of the unperturbed model. Thus
\begin{equation}\label{eq:newtondata}
 d(\phi)=\hlin(\phi)=\frac{mn}{m+1}=:d,
 \qquad
 h(\phi)=n,
 \qquad
 h^r(\phi)=\max\{d,n-1\},
\end{equation}
see \cite[Example~1.6 and Proposition~2.1]{IMI}.

For necessity, suppose that the restriction inequality in
	\eqref{eq:mainrestriction} holds for some $p\geq1$ and
	$\eta(\Gamma_\phi(0))>0$.
	Since $\partial_2^n\phi(0)=n!$, the phase is of finite type.
	The smooth necessity theorem \cite[Theorem~1.8]{IMI},
	which does not require Condition~\textup{(R)}, therefore gives
	\[
	p'\geq 2h^r(\phi)+2
	=\max\left\{2+\frac{2mn}{m+1},\,2n\right\}
	=\qc,
	\]
	where we used \eqref{eq:newtondata}.
	Equivalently, $p\leq\pc$.


\subsection{Extensions to more general phases}

\begin{theorem}
	\label{thm:extension-smooth-root}
	Let $m,n\ge2$ be integers. Suppose that
	\[
	a(0,0)\ne0,\qquad
	\psi(t)=ct^m+O(t^{m+1}),\qquad c\ne0,
	\]
	and that $R$ is flat at the origin. Define
	\[
	\phi(\xi_1,\xi_2)
	=
	a(\xi_1,\xi_2)
	\bigl(\xi_2-\psi(\xi_1)\bigr)^n
	+R(\xi_1,\xi_2).
	\]
Then there exists a sufficiently small bounded open neighborhood
	$U$ of the origin such that
	\[
	\|\mathcal E_{\phi,U}g\|_{L^q(\mathbb R^3)}
	\le C_q\|g\|_{L^2(U)}
	\qquad
	\bigl(q_c<q\le\infty\bigr).
	\]
	
	Let $S$ be the graph of $\phi$ over $U$. For every
	nonnegative $\eta\in C_c^\infty(S)$,
	\[
	\left(\int_S|\widehat f|^2\,d\mu\right)^{1/2}
	\le C_{p,\eta}\|f\|_{L^p(\mathbb R^3)}
	\qquad
	\left(
	1\le p<p_c
	\right),
	\]
where $\dd\mu=\eta\dd\sigma$ and $\dd\sigma$ denotes surface measure.	If $\eta(\Gamma_{\phi}(0))>0$, a necessary condition is
	$p\le p_c$.
\end{theorem}

The same method permits a prescribed tangential vanishing
factor in the coefficient.

\begin{theorem}
	\label{thm:extension-tangential}
	Let $m,n\ge2$ and $0\le k<n$ be integers. Suppose that
	\[
	a(0,0)\ne0,\qquad
	\psi(t)=ct^m+O(t^{m+1}),\qquad c\ne0,
	\]
	and that $R$ is flat at the origin. Define
	\[
	\phi(\xi_1,\xi_2)
	=
	\xi_1^k a(\xi_1,\xi_2)
	\bigl(\xi_2-\psi(\xi_1)\bigr)^n
	+R(\xi_1,\xi_2),
	\]
	and
	\[
d_k=\frac{k+mn}{m+1},
	\qquad
q_k=\max\{2+2d_k,\,2n\}.
	\]
Then	there exists a sufficiently small bounded open neighborhood
	$U$ of the origin such that
	\[
	\|\mathcal E_{\phi,U}g\|_{L^q(\mathbb R^3)}
	\le C_q\|g\|_{L^2(U)}
	\qquad
	\bigl(q_k<q\le\infty\bigr).
	\]
	For every nonnegative $\eta\in C_c^\infty(S)$, where $S$
	is the graph of $\phi$ over $U$, restriction holds for
	\[
	1\le p<\frac{q_k}{q_k-1}.
	\]
	If $\eta(\Gamma_{\phi}(0))>0$, it is necessary that
	\[
	p\le\frac{q_k}{q_k-1}.
	\]
	Thus the asserted range is optimal away from its critical
	exponent.
\end{theorem}

\begin{corollary}
	\label{cor:extension-defining-function}
	Suppose that
	\[
	F(0,0)=0,\qquad
	\partial_{2}F(0,0)\ne0.
	\]
	Let the local zero set of $F$ be $\xi_2=\psi(\xi_1)$, and
	assume that
	\[
	\psi(t)=ct^m+O(t^{m+1}),\qquad c\ne0,\quad m\ge2.
	\]
	If $a(0,0)\ne0$, $R$ is flat at the origin, and
	$n\ge2$, $0\le k<n$ are integers, then
	\[
	\phi(\xi)=\xi_1^ka(\xi)F(\xi)^n+R(\xi)
	\]
	satisfies all conclusions of
	Theorem~\ref{thm:extension-tangential}.
\end{corollary}

\subsection{Structure of the proof}
Section~\ref{sec:strip} proves the curved-strip estimate for $n\geq3$.
On a transverse shell of width $\delta$, the $TT^*$ argument gives an $L^2\to L^{2n}$ bound with loss $\delta^{-1/(2n)}$.
Restricting the tangential parameter to intervals of length $\delta^{1/2}$ yields the factor $\delta^{1/4-3/(4n)}$,
which cancels the principal planar decoupling loss.
Interpolation with the shell's $L^\infty$ bound then permits summation for every $q>2n$.

Section~\ref{sec:square} treats $n=2$ by a different argument.
A transverse slope coordinate and a bounded convolution density
with weight $\abs{u_1-u_2}$ give the required bilinear estimates.
The weighted bound follows by minimizing in a strictly convex
variable and using evenness and the substitution $w=s^2$.
Parabolic Whitney pairs contribute $\delta^{3-14/q}$, so their
sum converges for $q>14/3$. All estimates are uniform under
small $C^4$ perturbations of $(y-u^2)^2$.

Section~\ref{sec:global} uses Palle's exterior estimate
\cite[Proposition~4.3]{Palle} away from the principal root.
Near the root, anisotropic dilation reduces each annulus to a fixed
parameter region, where flatness supplies the uniform hypotheses
of the two local estimates. Scaling back and summing require
$q>2+2mn/(m+1)$. For $n=2$ this condition already implies
$q>14/3$; for $n\geq3$ it combines with $q>2n$ to give
\eqref{eq:critical}. The geometric smallness conditions determine
a single neighborhood $U=U(m,n,R)$, independent of $q$.

The proofs of Theorem \ref{thm:extension-smooth-root}, Theorem 
\ref{thm:extension-tangential}
 and Corollary \ref{cor:extension-defining-function} are given in
Section~\ref{sec:extensions}. They use the same uniform
local estimates as the main theorem; the tangential factor
changes the outer scaling exponent.

\subsection{Notation}

We use the Fourier transform
$\widehat f(\xi)=\int e^{-ix\cdot\xi}f(x)\dd x$
and absorb normalization constants into inequalities.
We write $A\lesssim B$ if $A\leq CB$, with dependence on
fixed parameters, domains, and geometric bounds understood;
subscripts indicate additional dependence.
The notation $A\approx B$ means that
$A\lesssim B\lesssim A$.

Unless another measure or weight is explicitly indicated,
$L^p$ norms use Lebesgue measure, and we sometimes omit the domains for simplicity.
For a nonnegative integer $N$, the $C^N$ norm is the maximum
of the supremum norms of all partial derivatives of order at most $N$.
For a function $F=F(z_1,\ldots,z_d)$, we write
$\partial_jF=\frac{\partial F}{\partial z_j}$ and $
 \partial^\alpha F=\partial_1^{\alpha_1}\cdots\partial_d^{\alpha_d}F$.
We also use the variable as subscripts;
for instance, $G_{st}=\partial_1\partial_2G$ for $G=G(s,t)$.

For sets $A,B\subset\R^d$ with $B$ open, $A\Subset B$ means
that $\overline A$ is compact and contained in $B$;
the set $A$ itself need not be open.

\subsection{AI usage}
We used ChatGPT 6.1 for English-language editing and proofreading, as well as to explore possible approaches to parts of the proofs. The principal ideas underlying the results were suggested by ChatGPT 6.1. The authors developed, carefully verified, and rewrote all the mathematical arguments, simplifying most of the proofs in the process. The authors are responsible for this paper.

\section{A uniform curved-strip estimate for \texorpdfstring{$n\geq3$}{n at least 3}}
\label{sec:strip}

The first local estimate concerns a strip around a curved zero set of
$\partial_2^{n-1}\varphi$ for a smooth real-valued phase $\varphi$.
Its proof separates the two parameter
directions. An oscillatory integral estimate controls the transverse
direction, while planar decoupling combines short pieces of the curved
strip. The two scale factors cancel at the exponent $2n$.

\begin{proposition}\label{prop:strip}
Let $n\geq3$, let $I$ be a fixed compact interval, and let $A$
be smooth on a neighborhood of $I$, with
\begin{equation}\label{eq:Aclass}
 \norm{A}_{C^3(I)}\leq M,
 \qquad \abs{A''(\xi_1)}\geq c_0>0\quad(\xi_1\in I).
\end{equation}
For a fixed $c>0$, consider the strip
\[
 U_{c}=\{(\xi_1,\xi_2):\xi_1\in I,\ \abs{\xi_2-A(\xi_1)}\leq c\}.
\]
Suppose that $\varphi$ is real-valued and smooth on a neighborhood
of this strip, and that
\begin{equation}\label{eq:derivativeconditions}
 \partial_2^{n-1}\varphi(\xi_1,A(\xi_1))=0,
 \qquad
 0<c_1\leq\partial_2^n\varphi(\xi_1,A(\xi_1)+v)\leq C_1
\end{equation}
for $\xi_1\in I$ and $\abs v\leq c$. Then, for every
$2n<q\leq\infty$,
\begin{equation}\label{eq:stripconclusion}
 \norm{\E_{\varphi,U_{c}}f}_{L^q(\R^3)}
 \leq C_q\norm{f}_{L^2(U_{c})}.
\end{equation}
The constant depends only on $q,n,I,c,M,c_0,c_1,C_1$.
\end{proposition}

\begin{proof}
The change of variables $\xi_1 =u, \xi_2=A(u)+v$ has unit Jacobian.
Set $\Phi(u,v)=\varphi(u,A(u)+v) = \varphi(\xi_1,\xi_2)$ and
$g(u,v)=f(u,A(u)+v)$. A simple calculation yields that $\norm g_{L^2(I\times[-c,c])}=\norm f_{L^2(U_c)}$ and
\begin{equation}\label{eq:stripoperator}
 \begin{aligned}
 \E g(x)
 &:=\int_I\int_{-c}^{c}
 e^{i[x_1u+x_2(A(u)+v)+x_3\Phi(u,v)]}g(u,v)\dd v\dd u\\
 &=\int_{U_c}
 e^{i[x_1\xi_1+x_2\xi_2+x_3\varphi(\xi_1,\xi_2)]}f(\xi_1,\xi_2)\dd\xi_1\dd\xi_2
 =\E_{\varphi,U_c}f(x).
 \end{aligned}
\end{equation}
The case $q=\infty$ follows from Cauchy--Schwarz inequality. 
Now we fix $2n<q<\infty$. On each side of $v=0$, decompose the transverse
interval into shells $J_\delta$ with
$\delta/2<\abs v\leq\delta$, where $\delta=2^{-j}c$,
$j\geq0$. Let $\E_\delta$ denote the restriction of
\eqref{eq:stripoperator} to $I\times J_\delta$.
We first establish an $L^{2n}$ bound on the shell estimate
with an arbitrarily small loss $\delta^{-\varepsilon}$.

\subsubsection*{1. The transverse estimate}

For each $u\in I$, define the operator $T_u$ as
\[
 (T_u h)(x_2,x_3)
 :=
 \int_{J_\delta}
 e^{i[x_2(A(u)+v)+x_3\Phi(u,v)]}h(v)\dd v,
 \qquad h\in L^2(J_\delta).
\]
Then
\[
 \E_\delta g(x)
 =
 \int_I e^{ix_1u}T_u[g(u,\cdot)](x_2,x_3)\dd u.
\]
We will prove 
\begin{equation}\label{eq:fiberbound}
 \norm{T_u[g(u,\cdot)]}_{L^{2n}(\R^2)}
 \lesssim\delta^{-\frac{1}{2n}}\norm{g(u,\cdot)}_{L^2(J_\delta)}
\end{equation}
 uniformly in $u$
by a $T_uT_u^*$ argument.
For fixed $u$, the kernel of $T_uT_u^*$ is
$K_{x_3-x_3'}(x_2-x_2')$. Writing $\lambda=x_3-x_3'$, set
\[
 K_\lambda(x_2)=\int_{J_\delta}  e^{i[x_2(A(u)+v)+\lambda\Phi(u,v)]}\dd v.
\]
The fundamental theorem of calculus gives
\begin{equation}\label{eq:rootidentity}
 \partial_2^{n-1}\Phi(u,v)
 =v\int_0^1\partial_2^n\varphi(u,A(u)+\theta v)\dd\theta.
\end{equation}
Thus, from \eqref{eq:derivativeconditions}, we have $\abs{\partial_2^{n-1}\Phi}\approx\delta$ on a single-sided
shell. 
Since $n-1\geq2$, differentiation removes the term linear in $v$, 
and the $(n-1)$-th derivative of the phase function in $K_\lambda$ has modulus
$\abs\lambda\,\abs{\partial_2^{n-1}\Phi(u,v)} \gtrsim\delta\abs\lambda$.

For $\lambda\neq0$, van der Corput lemma \cite[Chapter~VIII]{Stein} gives
\[
 \norm{K_\lambda*H}_\infty \lesssim(\delta\abs\lambda)^{-1/(n-1)}\norm H_1.
\]
With the Fourier convention of Section~1.4, the multiplier is
\[
 \widehat K_\lambda(\xi_2)
 =2\pi e^{i\lambda\Phi(u,\xi_2-A(u))}
          \mathbf1_{A(u)+J_\delta}(\xi_2).
\]
The expression is extended by zero outside the interval $A(u)+J_\delta$. Its modulus
is uniformly bounded, so Plancherel's theorem gives
$\norm{K_\lambda*H}_2\lesssim\norm H_2$.
Interpolating these two estimates, we have
\[
 \norm{K_\lambda*H}_{L^{2n}_{x_2}}
 \lesssim \delta^{-1/n}\abs\lambda^{-1/n}
                 \norm H_{L^{(2n)'}_{x_2}}.
\]
For $G\in C_c^\infty(\R^2)$, Minkowski inequality therefore gives
\[
 \norm{T_uT_u^*G(\cdot,x_3)}_{L^{2n}_{x_2}}
 \lesssim\delta^{-1/n}
 \int_\R\abs{x_3-x_3'}^{-1/n}\norm{G(\cdot,x_3')}_{L^{(2n)'}_{x_2}}\dd x_3'.
\]
Since $\frac{1}{2n}=\frac{1}{(2n)'}-(1-\frac{1}{n})$, 
the one-dimensional Hardy--Littlewood--Sobolev inequality \cite{SteinWeiss} yields
\[
 \norm{T_uT_u^*G}_{2n}\lesssim\delta^{-1/n}\norm G_{(2n)'}.
\]
Next, a simple calculation gives
\[
 \norm{T_u^*G}_2^2=\langle T_uT_u^*G,G\rangle \lesssim\delta^{-1/n}\norm G_{(2n)'}^2.
\]
Taking square roots and using duality proves \eqref{eq:fiberbound}. 

\subsubsection*{2. Short tangential pieces and decoupling}
For an interval $Q\subset I$ of length $\ell$, write
\[
 g_Q=g\mathbf1_{Q\times J_\delta},
 \qquad
 \E_{\delta,Q}g(x)=\int_Qe^{ix_1u} T_u[g(u,\cdot)](x_2,x_3)\dd u.
\]

For each fixed $(x_2,x_3)$, the Fourier support in $x_1$ of
$\E_{\delta,Q}g$ is contained in the interval $Q$,
whose length is $\ell$.
Applying the Bernstein's inequality yields that
\begin{equation*}
      \norm{\E_{\delta,Q}g}_{L^{2n}_{x}} \lesssim
 \ell^{\frac{1}{2}-\frac{1}{2n}}
 \left\| \norm{\E_{\delta,Q}g(\cdot,x_2,x_3)}_{L^2_{x_1}} \right\|_{L^{2n}_{x_2,x_3}}.
\end{equation*}
Using the Plancherel theorem
in $x_1$ and the Minkowski inequality, the norm on the right-hand side is bounded by
\begin{equation*}
 \left\|  \left(\int_Q \abs{T_u[g(u,\cdot)](x_2,x_3)}^2\dd u\right)^{\frac{1}{2}} \right\|_{L^{2n}_{x_2,x_3}}
\lesssim
\left( \int_Q \norm{T_u[g(u,\cdot)]}_{L^{2n}_{x_2,x_3}}^2\dd u \right)^{\frac{1}{2}}.
\end{equation*}

Thus,  we have
\begin{equation}\label{eq:capbound}
 \begin{aligned}
 \norm{\E_{\delta,Q}g}_{L^{2n}_{x}} 
 &\lesssim
 \ell^{\frac{1}{2}-\frac{1}{2n}} 
 \left( \int_Q \norm{T_u[g(u,\cdot)]}_{L^{2n}_{x_2,x_3}}^2\dd u \right)^{\frac{1}{2}} \\
&\lesssim
 \ell^{\frac{1}{2}-\frac{1}{2n}}\delta^{-\frac{1}{2n}}
 \left(
   \int_Q \norm{g(u,\cdot)}_{L^2(J_{\delta})}^2\dd u
 \right)^{1/2}
 =
 \ell^{\frac{1}{2}-\frac{1}{2n}}\delta^{-\frac{1}{2n}}\norm{g_Q}_2,
 \end{aligned}
\end{equation}
where the last inequality uses \eqref{eq:fiberbound}.

With $Q=I$, the above result already controls the finitely many shells with $\delta>1$. 
Actually, for $q \in [2n,\infty]$, interpolation the above estimate with the $L^\infty$ bound yields
\begin{equation}\label{eq:largedelta}
      \norm{\E_{\delta,I}g}_{L^{q}(\R^3)} \lesssim \norm{g}_{L^2(I\times J_\delta)}, 
      \quad \text{for}\ \delta>1,\ q\in [2n,\infty].
\end{equation} 
We may henceforth assume $0<\delta\leq1$.

The natural length of $Q$ is $\delta^{1/2}$: over such an interval,
the graph of $A$ differs from its tangent line by $O(\delta)$.
Partition $I$ at this scale, allowing a shorter endpoint interval.
Then \eqref{eq:capbound} becomes
\begin{equation}\label{eq:capbound2}
       \norm{\E_{\delta,Q}g}_{2n}
      \lesssim\delta^{1/4-3/(4n)}\norm{g_Q}_2.
\end{equation}

For each fixed $x_3$, the Fourier transform of $\E_\delta g$ in
$(x_1,x_2)$ is supported in
$\{(\xi_1,\xi_2):\xi_1\in I,\ \abs{\xi_2-A(\xi_1)}\leq\delta\}$.
It is an $L^2$ function with bounded support, so
$\E_\delta g(\cdot,\cdot,x_3)\in L^{2n}(\R^2)$.
The planar case of Bourgain--Demeter decoupling
\cite[Theorem~1.1]{BD} gives
\begin{equation}\label{eq:decouplingapplied}
 \norm{\E_\delta g(\cdot,\cdot,x_3)}_{L^{2n}_{x_1,x_2}}
 \lesssim_\varepsilon
 \delta^{-1/4+3/(4n)-\varepsilon}
 \left(\sum_Q
 \norm{\E_{\delta,Q}g(\cdot,\cdot,x_3)}_{L^{2n}_{x_1,x_2}}^2
 \right)^{1/2}.
\end{equation}

The constant in \eqref{eq:decouplingapplied} is uniform
under \eqref{eq:Aclass}.
For the case $n=3$, this follows from
\cite[Lemma~4.4]{YangUniform}, after normalizing $I$
and choosing a comparable dyadic scale.
For $n >3$, interpolate the corresponding cap summation
operator, using smooth cap cutoffs with uniformly
$L^1$-bounded convolution kernels and bounded overlap, between this $L^6$
bound and the trivial $\ell^2(L^\infty)\to L^\infty$
bound $(\#\{Q\})^{1/2}\lesssim\delta^{-1/4}$.
The interpolation weight $3/n$ gives the stated exponent,
after relabeling the arbitrarily small loss.

Taking the $L^{2n}$ norm in $x_3$, applying Minkowski inequality to the square sum, and using
\eqref{eq:capbound2}, we obtain
\begin{equation}\label{eq:shellcritical}
 \begin{aligned}
 \norm{\E_\delta g}_{2n}
 &\lesssim_\varepsilon
 \delta^{-1/4+3/(4n)-\varepsilon}
 \left(\sum_Q\norm{\E_{\delta,Q}g}_{2n}^2\right)^{1/2}\\
 &\lesssim_\varepsilon
 \delta^{-1/4+3/(4n)-\varepsilon}
 \delta^{1/4-3/(4n)}
 \left(\sum_Q\norm{g_Q}_2^2\right)^{1/2}\\
 &=\delta^{-\varepsilon}\norm g_{L^2(I\times J_\delta)}
 \leq\delta^{-\varepsilon}\norm g_2,
 \end{aligned}
\end{equation}
where the last step uses the disjointness of the intervals $Q$.

\subsubsection*{3. Summing the transverse shells}
The shell has parameter area $O(\delta)$, so
$\norm{\E_\delta g}_\infty\lesssim\delta^{1/2}\norm g_2$.
Interpolating this bound with \eqref{eq:shellcritical} yields
\begin{equation*}
 \norm{\E_\delta g}_q
 \lesssim_{q,\varepsilon}\delta^{\gamma(q,\varepsilon)}\norm g_2,
 \qquad
 \gamma(q,\varepsilon)=\frac12-\frac nq-\frac{2n\varepsilon}{q}.
\end{equation*}
Choose $\varepsilon=(q-2n)/(8n)$. Then $\gamma=(q-2n)/(4q)>0$. 
On either fixed side of $v=0$, let $g_j$ be the
restriction of $g$ to $I\times J_{\delta_j}$.
Then
\[
 \begin{aligned}
 \left\|\sum_{\delta_j\leq1}\E g_j\right\|_q
 \lesssim_q\sum_{\delta_j\leq1}\delta_j^\gamma\norm{g_j}_2
 \leq C_q\left(\sum_{\delta_j\leq1}\delta_j^{2\gamma}\right)^{1/2}
 \left(\sum_j\norm{g_j}_2^2\right)^{1/2}
 \lesssim_q\norm g_2.
 \end{aligned}
\]
The same estimate holds on the other side of $v=0$.
Combining the two sides with the finitely many coarse shells
in \eqref{eq:largedelta}, and using \eqref{eq:stripoperator},
proves \eqref{eq:stripconclusion}.
The shell expansion agrees with the original integral
because $g\in L^1$ on the bounded parameter domain, since
the omitted set $v=0$ has measure zero.
\end{proof}

\begin{remark}
      The proof imposes no additional conditions on the lower
transverse derivatives
$\partial_2^j\varphi(u,A(u))$, $0\leq j\leq n-2$.
The derivative lower bound required for the transverse
oscillatory estimate follows from \eqref{eq:rootidentity}
and the bounds on $\partial_2^n\varphi$, without any
factorization assumption on $\varphi$.
This flexibility will be used for the normalized flat perturbations.
\end{remark}

\section{Uniform stability for the squared parabolic root}
\label{sec:square}

For $n=2$, we estimate products of extensions by convolution.
A coordinate adapted to the transverse slope handles separation
in one parameter direction; a weighted convolution estimate handles
separation in the other. A parabolic Whitney decomposition then
reduces the extension estimate to a convergent sum over scales.
Set
\[
 P_0(u,y)=(y-u^2)^2,
 \qquad \Omega_*=(-3,3)\times(-10,10).
\]
The enlarged domain accommodates the two-point calculations below.
In the relevant smooth normal form of \cite[Theorem~5.1]{Palle},
the flat residual term $b_0$ is required to vanish identically.
The following estimate allows arbitrary small $C^4$ perturbations.

\begin{proposition}\label{prop:square}
There is an absolute $\varepsilon_0>0$ such that, if
\begin{equation}\label{eq:C4small}
 \varphi=P_0+H\quad\text{on }\Omega_*,
 \qquad
 \max_{\abs\alpha\leq4}\norm{\partial^\alpha H}_{L^\infty(\Omega_*)}
 \leq\varepsilon_0,
\end{equation}
where $H$ is real-valued and smooth, then on the patch
\[
 D_0=\{(u,y):\abs u\leq1,\ \abs{y-u^2}\leq1/2\}
\]
one has
\begin{equation*}
 \norm{\E_{\varphi,D_0}f}_{L^q(\R^3)}
 \leq C_q\norm f_{L^2(D_0)},
 \qquad 14/3<q\leq\infty.
\end{equation*}
The constants are uniform over the class \eqref{eq:C4small}.
\end{proposition}

\subsection{The transverse slope coordinate}
Take $\varepsilon_0\leq1/2$ and introduce $(u,v)$ by
\begin{equation*}
 v=\tfrac12\partial_2\varphi(u,y)
   =y-u^2+\tfrac12\partial_2H(u,y).
\end{equation*}
Since $3/4\leq\tfrac12\partial_2^2\varphi(u,y)\leq5/4$,
this change is strictly increasing in $y$ for each fixed $u$. For $\abs u\leq1$, its values at
$y=u^2\pm2$ lie beyond $\pm1$, respectively; these points
and the intervening segment lie in $\Omega_*$. Hence it has a
smooth inverse $y=Y_\varphi(u,v)$ on a neighborhood of
$D=[-1,1]^2$, with
\begin{equation}\label{eq:J}
 J(u,v):=\partial_2Y_\varphi(u,v)
 =\frac{2}{2+(\partial_2^2H)(u,Y_\varphi(u,v))},
 \qquad \frac45\leq J(u,v)\leq\frac43.
\end{equation}
Moreover, on $D_0$,
$\abs v\leq1/2+\varepsilon_0/2\leq3/4$, so its image lies in $D$.
Define
\begin{equation*}
 \Gamma(u,v)=\bigl(u,Y_\varphi(u,v),
          \varphi(u,Y_\varphi(u,v))\bigr),
 \qquad (u,v)\in D,
\end{equation*}
and write
\[
 \E g(x)=\int_D e^{ix\cdot\Gamma(u,v)}g(u,v)\dd u\dd v.
\]
Extending $f$ by zero outside $D_0$, the change of variables gives
$\E_{\varphi,D_0}f=\E g$ for
$g(u,v)=f(u,Y_\varphi(u,v))J(u,v)$, with
$\norm g_2\approx\norm f_2$ by \eqref{eq:J}.
It therefore suffices to prove
\begin{equation}\label{eq:squareestimate2}
 \norm{\E g}_{L^q(\R^3)}\leq C_q\norm g_{L^2(D)},
 \qquad 14/3<q\leq\infty,
\end{equation}
uniformly under \eqref{eq:C4small}.

\subsection{A bounded weighted convolution density}
\begin{lemma}\label{lem:weighted}
Define the finite positive measure $\nu$ by
\[
 \begin{aligned}
 \int_{\R^3}\Psi(\zeta)\dd\nu(\zeta)
 :=\int_{D\times D}&
 \Psi\bigl(\Gamma(u,v)+\Gamma(u',v')\bigr)\\
 &\times\abs{u-u'}\dd u\dd v\dd u'\dd v',
 \qquad \Psi\in C_c(\R^3).
 \end{aligned}
\]
Then $\nu=W\dd\zeta$, where
\begin{equation*}
 \norm W_{L^\infty(\R^3)}\leq C
\end{equation*}
with an absolute constant.
\end{lemma}

\begin{proof}
Return to the graph coordinates $(u,y)$ and $(u',y')$, where
$y=Y_\varphi(u,v)$ and $y'=Y_\varphi(u',v')$.
The Jacobian factors $\tfrac12(\partial_2^2\varphi)(u,y)$ and
$\tfrac12(\partial_2^2\varphi)(u',y')$ are uniformly bounded
and may be suppressed in an upper bound.
Fix $u+u'=\zeta_1=2a$ and $y+y'=\zeta_2=2c$, and write
\begin{equation*}
 u=a+s,\quad u'=a-s,
 \qquad y=c+t,\quad y'=c-t.
\end{equation*}
For $(u,v),(u',v')\in D$, we have $\abs a,\abs s\leq1$ and
$\max\{\abs y,\abs{y'}\}\leq2+\varepsilon_0/2<3$,
hence $\abs c,\abs t<3$.
For these $a,c$, the full rectangle $\abs s\leq1$, $\abs t\leq4$
still gives points in $\Omega_*$, since
$\abs{a\pm s}\leq2$ and $\abs{c\pm t}<7$.
The phase sum and its unperturbed counterpart are
\begin{equation*}
 \begin{aligned}
 G(s,t)&=\varphi(a+s,c+t)+\varphi(a-s,c-t),\\
 G_0(s,t)&=2(c-a^2-s^2)^2+2(t-2as)^2.
 \end{aligned}
\end{equation*}
Every partial derivative of $G-G_0$ of order at most four is
bounded by $2\varepsilon_0$. In particular, $G_{tt}\geq3$.

The function $G_t(s,\cdot)$ has opposite signs at $2as\pm1$.
It therefore has a unique zero $t=\tau(s)$ in $[-4,4]$, and
\begin{equation*}
 \abs{\tau(s)-2as}\leq\varepsilon_0/2.
\end{equation*}
Set $F(s)=G(s,\tau(s))$. The symmetry $G(-s,-t)=G(s,t)$ and
uniqueness imply that $\tau$ is odd and $F$ is even.
We next show that $F^{(4)}$ stays positive using only four
derivatives of $G$.
All partial derivatives in the following calculation are evaluated
at $(s,\tau(s))$. Differentiating $G_t(s,\tau(s))=0$ gives
\[
 h:=\tau'=-\frac{G_{st}}{G_{tt}},
 \qquad
 h'=-\frac{G_{sst}+2hG_{stt}+h^2G_{ttt}}{G_{tt}},
 \qquad \abs h\leq3.
\]
Starting from $F'=G_s$, we obtain
\[
 F'''=G_{sss}+3hG_{sst}+3h^2G_{stt}+h^3G_{ttt},
\]
and one more differentiation yields
\begin{equation*}
 \begin{aligned}
 F^{(4)}={}&G_{ssss}+4hG_{ssst}+6h^2G_{sstt}
              +4h^3G_{sttt}+h^4G_{tttt}\\
 &-\frac{3}{G_{tt}}
       \bigl(G_{sst}+2hG_{stt}+h^2G_{ttt}\bigr)^2.
 \end{aligned}
\end{equation*}
For $G_0$, we have $(G_0)_{ssss}=48$, while the other fourth-order
derivatives and all third-order derivatives involving $t$ vanish.
Thus, uniformly in $a,c$,
\begin{equation}\label{eq:Ffourlower}
 F^{(4)}(s)=48+O(\varepsilon_0)\geq24,
\end{equation}
after choosing $\varepsilon_0$ sufficiently small.

Put $K(w)=F(\sqrt w)$ for $0\leq w\leq1$.
Since $F$ is even, we have $F'(0)=F'''(0)=0$. Thus, for $s=\sqrt w>0$, we have
\begin{equation*}
 K''(w)=\frac{sF''(s)-F'(s)}{4s^3}
       =\frac1{4s^3}\int_0^s \theta F'''(\theta)\dd\theta\geq2.
\end{equation*}
Here $F'''(\theta)\geq24\theta$ by \eqref{eq:Ffourlower}.
Taylor's theorem gives $K''(0)=F^{(4)}(0)/12$, so $K\in C^2([0,1])$
and the same lower bound holds at zero.

For fixed $s$, strict convexity in $t$ gives at most two preimages
of any value $\zeta_3$. At a regular preimage,
\begin{equation*}
 \abs{G_t(s,t)}^2\geq6\bigl(\zeta_3-F(s)\bigr).
\end{equation*}
Indeed, integrate $\partial_t(G_t^2)=2G_tG_{tt}$ between
$t$ and $\tau(s)$, using $G_{tt}\geq3$ on either side of the minimum.
Changing variables on the two monotone branches bounds the resulting
density by $C(\zeta_3-F(s))^{-1/2}$ for $\zeta_3>F(s)$.
The critical graph $t=\tau(s)$ has measure zero, so these changes
of variables and Fubini also prove absolute continuity.
Since $\abs{u-u'}=2\abs s$, evenness and $w=s^2$ give
\begin{equation*}
 \begin{aligned}
 W(2a,2c,\zeta_3)
 &\lesssim\int_{\{s\in[-1,1]:F(s)<\zeta_3\}}
       \frac{\abs s\dd s}{\sqrt{\zeta_3-F(s)}}\\
 &=\int_{\{w\in[0,1]:K(w)<\zeta_3\}}
       \frac{\dd w}{\sqrt{\zeta_3-K(w)}}
 \end{aligned}
\end{equation*}
for almost every $(2a,2c,\zeta_3)\in\R^3$.
If the last sublevel set is nonempty, its closure is an interval
$[\ell_-,\ell_+]$, with $K(\ell_-),K(\ell_+)\leq \zeta_3$, also at truncated endpoints.
By $K''\geq2$,
\[
 K(w)\leq\frac{\ell_+-w}{\ell_+-\ell_-}K(\ell_-)
          +\frac{w-\ell_-}{\ell_+-\ell_-}K(\ell_+)
          -(w-\ell_-)(\ell_+-w).
\]
Consequently,
\begin{equation*}
 \int_{\ell_-}^{\ell_+}\frac{\dd w}{\sqrt{\zeta_3-K(w)}}
 \leq\int_{\ell_-}^{\ell_+}\frac{\dd w}{\sqrt{(w-\ell_-)(\ell_+-w)}}=\pi,
\end{equation*}
which proves the lemma.
\end{proof}

\subsection{Whitney pairs and bilinear estimates}

We use the parabolic Whitney decomposition associated with the
nested grid on $[-2,2]\times[-8,8]$; compare
\cite[Section~2.5]{TVV}. At scale $\delta_j=4\cdot2^{-j}$,
$j\geq1$, the rectangles have side lengths
$\delta_j\times\delta_j^2$. Write $Q\sim Q'$ when their
closures are disjoint but their parents' closures intersect;
the parent has side lengths $2\delta_j\times4\delta_j^2$.
The relation is defined before intersecting the rectangles with $D$.

At each scale the rectangles partition $D$, and each has $O(1)$
partners. Over all scales, the ordered products $Q\times Q'$
partition $D\times D$ off the diagonal, up to null sets.
For each pair, writing $\delta=\delta_j$, one of the bounds
\begin{equation*}
 \abs{u-u'}\geq c\delta
 \qquad\text{or}\qquad
 \abs{v-v'}\geq c\delta^2
\end{equation*}
holds throughout $Q\times Q'$.

Let $\sigma_Q$ be the finite positive measure on $\R^3$
defined by
\[
 \int_{\R^3}\Psi(\zeta)\dd\sigma_Q(\zeta)
 =\int_Q\Psi(\Gamma(u,v))\dd u\dd v,
 \qquad \Psi\in C_c(\R^3).
\]
Thus $\sigma_Q$ assigns to a Borel set $A\subset\R^3$
the parameter area of
$\{(u,v)\in Q:\Gamma(u,v)\in A\}$.
We claim that $\sigma_Q*\sigma_{Q'}$ has a density satisfying
\begin{equation}\label{eq:pairdensity}
 \norm{\sigma_Q*\sigma_{Q'}}_\infty\lesssim\delta^{-1}
 \qquad(Q\sim Q').
\end{equation}
For $u$-separated pairs, positivity and Lemma~\ref{lem:weighted}
give $\sigma_Q*\sigma_{Q'}\leq(c\delta)^{-1}\nu$.
For $v$-separated pairs, return to $(u,y)$ and $(u',y')$.
Fix $u+u'=\zeta_1$ and $y+y'=\zeta_2$, so that
$u'=\zeta_1-u$ and $y'=\zeta_2-y$.
The derivative of the third-coordinate phase in $y$ is
\begin{equation*}
 (\partial_2\varphi)(u,y)-(\partial_2\varphi)(u',y')=2(v-v'),
\end{equation*}
whose magnitude is at least $c\delta^2$ at contributing points.
The second derivative of the phase sum in $y$ is
$(\partial_2^2\varphi)(u,y)+(\partial_2^2\varphi)(u',y')\geq3$,
so each level has at most two preimages. Changing variables
in $y$ costs at most $C\delta^{-2}$, and the remaining $u$-interval
has length at most $C\delta$. This proves \eqref{eq:pairdensity}.

For $f\in L^1(Q)$, write $f\sigma_Q$ for the finite
complex measure defined by
\[
 \int_{\R^3}\Psi(\zeta)\dd(f\sigma_Q)(\zeta)
 =\int_Q\Psi(\Gamma(u,v))f(u,v)\dd u\dd v,
 \qquad \Psi\in C_c(\R^3).
\]
In particular,
\[
 \E_Qf(x)
 =\int_{\R^3}e^{ix\cdot\zeta}\dd(f\sigma_Q)(\zeta).
\]
For $f\in L^2(Q)$ and $h\in L^2(Q')$, all the convolution
measures below are absolutely continuous; we identify them with
their densities. Applying Cauchy--Schwarz inequality to the preceding
change-of-variables formulas gives
\[
 \abs{(f\sigma_Q)*(h\sigma_{Q'})(\zeta)}^2
 \leq(\sigma_Q*\sigma_{Q'})(\zeta)
       \bigl((\abs f^2\sigma_Q)*(\abs h^2\sigma_{Q'})\bigr)(\zeta).
\]
Integration and Plancherel therefore yield
\begin{equation*}
 \norm{\E_Qf\,\E_{Q'}h}_2
 \lesssim\delta^{-1/2}\norm f_2\norm h_2.
\end{equation*}
Since $\abs Q,\abs{Q'}\leq\delta^3$, Cauchy--Schwarz inequality also gives
\begin{equation*}
 \norm{\E_Qf\,\E_{Q'}h}_\infty
 \lesssim\delta^3\norm f_2\norm h_2.
\end{equation*}
Interpolating these norms, with weight $4/q$ on the $L^2$ bound,
we obtain
\begin{equation}\label{eq:bilineargain}
 \norm{\E_Qf\,\E_{Q'}h}_{q/2}
 \lesssim\delta^{3-14/q}\norm f_2\norm h_2,
 \qquad 4\leq q<\infty.
\end{equation}

\subsection{Completion of the proof}
\begin{proof}[Proof of Proposition~\ref{prop:square}]
Fix $14/3<q<\infty$ and $g\in L^2(D)$.
The Whitney decomposition gives
\[
 (\E g)^2=\sum_{j\geq1}\sum_{Q\sim Q'}
          \E_Qg_Q\,\E_{Q'}g_{Q'},
 \qquad g_Q=g\mathbf1_Q.
\]
Since $g\in L^1(D)$, the product expansion follows from
the Fubini theorem. We first estimate finite partial sums;
the fact that $\sum_j\delta_j^{3-14/q}<\infty$ give absolute convergence in $L^{q/2}$
and identify the limit with $(\E g)^2$.

At each scale the rectangles partition $D$ and have boundedly
many partners, so
\[
 \sum_{Q\sim Q'}\norm{g_Q}_2\norm{g_{Q'}}_2
 \lesssim\sum_Q\norm{g_Q}_2^2=\norm g_2^2.
\]
The triangle inequality in $L^{q/2}$ and \eqref{eq:bilineargain}
now give
\begin{equation*}
 \norm{\E g}_q^2
 \lesssim\sum_{j\geq1}\delta_j^{3-14/q}\norm g_2^2
 \lesssim_q\norm g_2^2,
\end{equation*}
where we use the fact $3-14/q>0$ to make the geometric series converge.
This proves \eqref{eq:squareestimate2}, and \eqref{eq:J}
transfers the bound to $D_0$.
The case $q=\infty$ follows directly from Cauchy--Schwarz inequality on $D_0$.
\end{proof}

\section{Proof of the flat-perturbation theorem}
\label{sec:global}

Let $\phi (\xi_1,\xi_2) = (\xi_2-\xi_1^m)^n + R(\xi_1,\xi_2)$, where $m,n \ge 2$ and $R$ satisfies \eqref{eq:flat}.
Let $D_\phi$ denote the original parameter domain of $\phi$.
We shall choose $U=U(m,n,R)\Subset D_\phi$, independently of
$q$, and prove \eqref{eq:mainextension} for
\[
 \E_{\phi,U}g(x)
 :=\int_U e^{i(x_1\xi_1+x_2\xi_2+x_3\phi(\xi_1,\xi_2))}g(\xi_1,\xi_2)\dd\xi_1\dd\xi_2,
 \qquad g\in L^2(U).
\]
The neighborhood $U$ is fixed after the geometric smallness
conditions below, before choosing $q$ or $g$.
Functions are extended by zero outside their parameter domains.
Rescaled coordinates are denoted by
$(\widetilde u,\widetilde y,\widetilde z)$, and
$x\in\R^3$ denotes the Fourier variable.

\subsection{Separate the exterior from the root neighborhood}
The Newton data \eqref{eq:newtondata} show that the original coordinates
are linearly adapted but not adapted, with leading principal root
$\xi_2=\xi_1^m$. Also,
$\partial_2^n\phi=n!+\partial_2^nR$ stays nonzero on a small
neighborhood, so the phase is of finite type there.

Choose $\chi\in C_c^\infty(\R)$ with $0\leq\chi\leq1$,
$\chi=1$ on $[-1,1]$, and $\supp\chi\subset[-2,2]$, and set
\begin{equation*}
 \rho(\xi_1,\xi_2)=\chi\!\left(\frac{\xi_2-\xi_1^m}{\eta_0\xi_1^m}\right)
 \qquad(\xi_1\neq0).
\end{equation*}
Here $\eta_0>0$ will be fixed sufficiently small below; the values
of $\rho$ on $\xi_1=0$ do not affect any integral.
For each sufficiently small fixed $\eta_0$, Palle's exterior
estimate \cite[Proposition~4.3]{Palle}, applied separately on the
two sides of $\xi_1=0$ with a smooth cutoff equal to one on the
neighborhood under consideration, provides a bounded open neighborhood
$U_{\mathrm{ext}}\Subset D_\phi$ of the origin such that
\[
 \norm{\E_{\phi,U}((1-\rho)g)}_q
 \lesssim_q \norm g_{L^2(U,(1-\rho)\dd\xi)}
 \leq \norm g_{L^2(U)}
\]
whenever $U\subset U_{\mathrm{ext}}$, $g\in L^2(U)$, and
\begin{equation*}
 2d+2=2+\frac{2mn}{m+1}\leq q\leq\infty.
\end{equation*}
More precisely, the neighborhood $U_{\mathrm{ext}}$ is chosen at $q=2d+2$;
interpolation with the trivial $L^\infty$ bound gives the remaining
exponents on the same neighborhood.

It remains to estimate $\E_{\phi,U}(\rho g)$.
Since $\rho g$ is supported in the horn
\begin{equation*}
 \abs{\xi_2-\xi_1^m}\leq2\eta_0\abs{\xi_1}^m,
\end{equation*}
we first establish estimates on its normalized annuli, and then
choose $U\subset U_{\mathrm{ext}}$ small enough for them to apply.

\subsection{Rescale the annuli and use flatness}
Split the horn into disjoint annuli $r\leq\abs{\xi_1}<2r$, with
$r$ dyadic. Writing $\xi_3=\phi(\xi_1,\xi_2)$ on the graph,
the dilation in graph coordinates is
\begin{equation*}
 \xi_1=r\widetilde u,\qquad \xi_2=r^m\widetilde y,
 \qquad \xi_3=r^{mn}\widetilde z.
\end{equation*}
Writing $S_r(\widetilde u,\widetilde y)
=(r\widetilde u,r^m\widetilde y)$ for its parameter part,
we obtain the normalized phase
\begin{equation*}
 \begin{aligned}
 \phi_r(\widetilde u,\widetilde y)
 &=r^{-mn}\phi\bigl(S_r(\widetilde u,\widetilde y)\bigr)
   =(\widetilde y-\widetilde u^m)^n
       +R_r(\widetilde u,\widetilde y),\\
 R_r(\widetilde u,\widetilde y)
 &=r^{-mn}R(r\widetilde u,r^m\widetilde y).
 \end{aligned}
\end{equation*}
Set $I_+=[1,2]$ and $I_-=[-2,-1]$.
Every normalized annulus is contained in the fixed compact set
\begin{equation*}
 \Omega:=\bigl\{(\widetilde u,\widetilde y):
 \widetilde u\in I_+\cup I_-,\quad
 \abs{\widetilde y-\widetilde u^m}
       \leq2\eta_0\abs{\widetilde u}^{m}\bigr\}.
\end{equation*}
Write $\Omega_\pm=\Omega\cap\{\pm\widetilde u>0\}$.
For all sufficiently small $r$, define
\[
 \E_{\phi_r,\Omega}G(x)
 :=\int_\Omega
 e^{i(x_1\widetilde u+x_2\widetilde y
                         +x_3\phi_r(\widetilde u,\widetilde y))}
 G(\widetilde u,\widetilde y)\dd\widetilde u\dd\widetilde y,
 \qquad G\in L^2(\Omega).
\]
This operator uses parameter Lebesgue measure, with no scaling
factor included in its definition.

On any fixed compact set $\mathcal K$ in the
$(\widetilde u,\widetilde y)$-plane,
$S_r(\mathcal K)\Subset D_\phi$ and
$\abs{S_r(\widetilde u,\widetilde y)}\lesssim_{\mathcal K}r$
for all sufficiently small $r$.
Since every derivative of $R$ is flat at the origin, Taylor's theorem and the chain rule give, for $i+j\leq N$ and every
integer $L\geq0$,
\[
 \begin{aligned}
 \abs{(\partial_1^i\partial_2^jR_r)
       (\widetilde u,\widetilde y)}
 =r^{i+mj-mn}
   \abs{(\partial_1^i\partial_2^jR)
                  (r\widetilde u,r^m\widetilde y)}
 \lesssim_{N,L,\mathcal K}r^{L+i+mj-mn}.
 \end{aligned}
\]
Choosing $L\geq mn+K$ yields
\begin{equation}\label{eq:rapidC}
 \norm{R_r}_{C^N(\mathcal K)}
 \leq C_{N,K,\mathcal K}r^K,
 \qquad N,K\geq0.
\end{equation}
Thus $\phi_r$ converges to $(\widetilde y-\widetilde u^m)^n$
faster than any power of $r$ in every fixed $C^N(\mathcal K)$ norm.

\subsection{The normalized estimate when \texorpdfstring{$n\geq3$}{n at least 3}}
The derivative-zero curve required by Proposition~\ref{prop:strip}
is given by
\begin{equation}\label{eq:Ar}
 \begin{aligned}
 \partial_2^{n-1}
       \phi_r(\widetilde u,A_r(\widetilde u))&=0,\\
 n!\bigl(A_r(\widetilde u)-\widetilde u^m\bigr)
 &=-\partial_2^{n-1}
       R_r(\widetilde u,A_r(\widetilde u)).
 \end{aligned}
\end{equation}
The derivative of the defining equation in $A_r$ tends uniformly
to $n!$. The implicit-function theorem and \eqref{eq:rapidC}
therefore give a unique nearby solution on slightly larger intervals
than $I_\pm$, with
\begin{equation*}
 A_r(\widetilde u)\longrightarrow\widetilde u^m
 \quad\text{in every fixed }C^N\text{ norm}.
\end{equation*}
The derivative bounds follow by successively differentiating
\eqref{eq:Ar}, whose coefficient of the highest derivative of
$A_r$ is bounded away from zero.

On $I_+\cup I_-$, the absolute value of the second derivative of
$\widetilde u^m$ is at least $m(m-1)$. Hence
$\abs{A_r''}$ has a uniform positive lower bound and
$\norm{A_r}_{C^3}$ a uniform upper bound.
Fix a small strip width $c_*>0$ and put
\[
 \mathcal V_{r,\pm}:=
 \{(\widetilde u,\widetilde y):\widetilde u\in I_\pm,
             \ \abs{\widetilde y-A_r(\widetilde u)}\leq c_*\}.
\]
On these strips,
$\partial_2^n\phi_r
=n!+\partial_2^nR_r$
has positive uniform upper and lower bounds for small $r$.
Choose $\eta_0$ small and then $r$ small so that
$\Omega_\pm\subset\mathcal V_{r,\pm}$.
Apply Proposition~\ref{prop:strip} after extending
$G\mathbf1_{\Omega_\pm}$ by zero to $\mathcal V_{r,\pm}$,
and add the two estimates. This gives
\begin{equation}\label{eq:annulusn3}
 \norm{\E_{\phi_r,\Omega}G}_{L^q(\R^3)}
 \leq C_q\norm G_{L^2(\Omega)},
 \qquad 2n<q\leq\infty,
\end{equation}
with $C_q$ independent of $r$. The sign of the curvature on either
interval is immaterial.

\subsection{The normalized estimate when \texorpdfstring{$n=2$}{n equal to 2}}
Let $A(\widetilde u)=\widetilde u^m$ and cover $I_+\cup I_-$
by finitely many patches $\abs{\widetilde u-\widetilde u_0}\leq h$.
Here $h>0$ is chosen independently of $r$.
For each patch, put $c=A''(\widetilde u_0)/2\neq0$ and make
the affine change from local graph coordinates $(u,y,z)$
\begin{equation*}
 \begin{aligned}
 \widetilde u&=\widetilde u_0+hu,\\
 \widetilde y&=A(\widetilde u_0)+A'(\widetilde u_0)hu+ch^2y,\\
 \widetilde z&=c^2h^4z.
 \end{aligned}
\end{equation*}
Let $T(u,y)$ denote the parameter part of this change.
The transformed phase is
\begin{equation}\label{eq:localphase}
 \begin{aligned}
 \widetilde\phi_r(u,y)
 &:=\frac{\phi_r(T(u,y))}{c^2h^4}\\
 &=\left(y-
 \frac{A(\widetilde u_0+hu)-A(\widetilde u_0)
                       -A'(\widetilde u_0)hu}{ch^2}\right)^2
   +\frac{R_r(T(u,y))}{c^2h^4}.
 \end{aligned}
\end{equation}
Apply Proposition~\ref{prop:square} in $(u,y)$ coordinates,
with its enlarged domain $\Omega_*=(-3,3)\times(-10,10)$
and patch $D_0=\{(u,y):\abs u\leq1,\ \abs{y-u^2}\leq1/2\}$.
Thus the enlarged domain in normalized coordinates is
$T(\Omega_*)$ in the $(\widetilde u,\widetilde y)$-plane.

Taylor's theorem for the fixed polynomial $A$ gives
\begin{equation*}
 \frac{A(\widetilde u_0+hu)-A(\widetilde u_0)
                       -A'(\widetilde u_0)hu}{ch^2}
 =u^2+O_{C^4(\Omega_*)}(h),
\end{equation*}
uniformly in the centers.
Choose $h$ small enough that the first term of
\eqref{eq:localphase} differs from $(y-u^2)^2$ by at most
$\varepsilon_0/2$ in $C^4(\Omega_*)$.
With $h$ fixed, the finitely many sets $T(\overline{\Omega_*})$
lie in a fixed compact normalized set. By \eqref{eq:rapidC}
and the chain rule, the second term is smaller than
$\varepsilon_0/2$ in the same norm for every sufficiently small $r$.
Thus $\widetilde\phi_r$ satisfies the hypothesis of
Proposition~\ref{prop:square} on $\Omega_*$.

For each point in the current patch of $\Omega$, the corresponding
local coordinates $(u,y)$ satisfy $\abs u\leq1$ and
\begin{equation*}
 \abs{y-u^2}\leq C\eta_0h^{-2}+Ch.
\end{equation*}
Choose $h$ also so that $Ch\leq1/4$, and then choose
$\eta_0$ so that $C\eta_0h^{-2}\leq1/4$, keeping it below the
threshold for the exterior theorem.
Thus the inverse image of this patch under $T$ lies in
$D_0\subset\Omega_*$.

The parameter Jacobian and the linear part of the ambient map are
\[
 J_{\mathrm{par}}=\abs c\,h^3,
 \qquad
 L=\begin{pmatrix}
 h&0&0\\
 A'(\widetilde u_0)h&ch^2&0\\
 0&0&c^2h^4
 \end{pmatrix},
 \qquad \abs{\det L}=\abs c^3h^7.
\]
For a function $G_{\mathrm{patch}}$ supported in this patch, set
$F=G_{\mathrm{patch}}\circ T$, extended by zero on $D_0$,
and $b=(\widetilde u_0,A(\widetilde u_0),0)$.
The change of variables gives
\[
 \E_{\phi_r,\Omega}G_{\mathrm{patch}}(x)
 =J_{\mathrm{par}}e^{ix\cdot b}
       \E_{\widetilde\phi_r,D_0}F(L^Tx),
 \qquad
 \norm F_{L^2(D_0)}
 =J_{\mathrm{par}}^{-1/2}
       \norm{G_{\mathrm{patch}}}_{L^2(\Omega)}.
\]
Consequently, the total factor in the extension estimate is
\begin{equation*}
 J_{\mathrm{par}}^{1/2}\abs{\det L}^{-1/q}
 =\abs c^{1/2-3/q}h^{3/2-7/q}.
\end{equation*}
Since $h$ is fixed and $\abs c$ has positive uniform upper and
lower bounds, this factor is independent of $r$.
Decompose $G$ into a finite sum of functions supported in the
respective patches, apply Proposition~\ref{prop:square} to each
term, and sum the resulting estimates.
We obtain
\begin{equation}\label{eq:annulusn2}
 \norm{\E_{\phi_r,\Omega}G}_{L^q(\R^3)}
 \leq C_q\norm G_{L^2(\Omega)},
 \qquad 14/3<q\leq\infty,
\end{equation}
uniformly for small $r$.
The choices are made in the order $h$, $\eta_0$, and finally the
upper bound on $r$, independently of $q$.

\subsection{Return to the original scale and sum}
Choose a fixed compact rectangle $\mathcal K$ containing $\Omega$
and the enlarged domains used above. Only finitely many derivative
bounds, of order at most $n+2$, enter the local hypotheses.
By \eqref{eq:rapidC}, there is $r_0=r_0(m,n,R)\in(0,1)$ such
that $S_r(\mathcal K)\Subset D_\phi$ and those hypotheses hold
for every $0<r<r_0$. Fix a bounded open neighborhood of the origin
satisfying
\[
 U\Subset U_{\mathrm{ext}}
       \cap\{(\xi_1,\xi_2)\in\R^2:\abs{\xi_1}<r_0\}.
\]
This choice is independent of $q$.

Now fix a finite $q>\qc$ and $g\in L^2(U)$.
For $n\geq3$, use \eqref{eq:annulusn3}; for $n=2$, use
\eqref{eq:annulusn2}, since
\[
 q>2+\frac{4m}{m+1}\geq\frac{14}{3}.
\]
For dyadic $0<r<r_0$, set
\[
 g_r(\xi_1,\xi_2)=\rho(\xi_1,\xi_2)g(\xi_1,\xi_2)\mathbf1_{\{r\leq\abs{\xi_1}<2r\}},
 \qquad
 G_r(\widetilde u,\widetilde y)
     =g_r\bigl(S_r(\widetilde u,\widetilde y)\bigr),
\]
using zero extension outside $U$.
Then $G_r=0$ almost everywhere outside
$\Omega\cap S_r^{-1}(U)$ and $\sum_r g_r=\rho g$ almost
everywhere on $U$.

The exact scaling identity is
\begin{equation*}
 \E_{\phi,U}g_r(x_1,x_2,x_3)
 =r^{m+1}\E_{\phi_r,\Omega}G_r
       (rx_1,r^mx_2,r^{mn}x_3).
\end{equation*}
The Fourier-variable change has determinant $r^{mn+m+1}$, and
$\norm{G_r}_{L^2(\Omega)}
 =r^{-(m+1)/2}\norm{g_r}_{L^2(U)}$.
Hence
\begin{equation*}
 \norm{\E_{\phi,U}g_r}_q
 \leq C_qr^{\alpha(q)}\norm{g_r}_{L^2(U)},
 \qquad
 \alpha(q)=\frac{m+1}{2}-\frac{mn+m+1}{q}>0.
\end{equation*}
Positivity is exactly the condition $q>2+2mn/(m+1)$.
Disjointness of the annuli and Cauchy--Schwarz inequality give
\begin{equation*}
 \begin{aligned}
 \norm{\E_{\phi,U}(\rho g)}_q
 &\leq C_q\sum_r r^{\alpha(q)}\norm{g_r}_{L^2(U)}\\
 &\leq C_q\left(\sum_r r^{2\alpha(q)}\right)^{1/2}
          \left(\sum_r\norm{g_r}_{L^2(U)}^2\right)^{1/2}
 \lesssim_q\norm g_{L^2(U)}.
 \end{aligned}
\end{equation*}
The annular expansion also agrees with the integral, since
$g\in L^1(U)$. Adding the exterior contribution proves
\eqref{eq:mainextension}.
The line $\xi_1=0$ has parameter measure zero, and the case
$q=\infty$ follows directly from Cauchy--Schwarz inequality on $U$.

\subsection{Surface measure and duality}
Write $\eta\dd\sigma=a(\xi_1,\xi_2)\dd\xi_1\dd\xi_2$ under the graph parametrization,
where
\[
 a(\xi_1,\xi_2)=\eta(\Gamma_\phi(\xi_1,\xi_2))\sqrt{1+\abs{\nabla\phi(\xi_1,\xi_2)}^2}
\]
is bounded and nonnegative on $U$.
Duality in \eqref{eq:mainextension} gives
\[
 \left(\int_S\abs{\widehat f}^{\,2}\eta\dd\sigma\right)^{1/2}
 \leq\norm a_\infty^{1/2}
       \norm{\widehat f\circ\Gamma_\phi}_{L^2(U)}
 \lesssim_{q,\eta}\norm f_{L^{q'}(\R^3)}.
\]
This proves \eqref{eq:mainrestriction}, including $p=1$
by $q=\infty$.

\section{Proofs of the extensions}
\label{sec:extensions}

We prove Theorem~\ref{thm:extension-tangential}.
The case $k=0$ gives
Theorem~\ref{thm:extension-smooth-root}. 
The proof uses Proposition~\ref{prop:strip},
Proposition~\ref{prop:square}, and the exterior restriction
theorem stated as \cite[Proposition~4.3]{Palle}.

Next, we always assume $k\not= 0$. Set $\nu=k+mn$, then $d_k=\frac{\nu}{m+1}$ and $q_k=\max\{2+2d_k,\,2n\}$.

\subsection{Newton data and the exterior region}

Since $\psi$ is smooth, we can write
\[
\psi(t)=t^m(c+t\omega(t))
\]
with $\omega$ smooth. Relative to weights $(1,m)$, the
principal part of the phase is
\[
a(0,0)\xi_1^k(\xi_2-c\xi_1^m)^n.
\]
Every contribution from $a-a(0,0)$ or from
$\psi-c\xi_1^m$ has strictly larger weighted degree.
The flat term $R$ contributes no Taylor coefficients.

The principal edge joins $(k,n)$ and $(k+mn,0)$.
Since $k<n$, the diagonal meets this edge and
\begin{equation}
	\label{eq:extension-D}
	d(\phi)=d_k=\frac{k+mn}{m+1},
	\qquad k<d_k<n.
\end{equation}

We verify linear adaptedness.
The lowest ordinary homogeneous term is
$a(0,0)\xi_1^k\xi_2^n$.
Under an invertible linear substitution $\xi=L\zeta$,
suppose first that the linear form representing $\xi_2$
has two nonzero coefficients.
If the linear form representing $\xi_1$ also has two
nonzero coefficients, both pure monomials of degree $k+n$
occur. If it has only one nonzero coefficient, the
endpoint exponent vectors are
$(k+n,0),(k,n)$, or their coordinate interchanges.
Because $k<n$, these give Newton distance at most
$(k+n)/2$. Furthermore,
\[
d_k-\frac{k+n}{2}
=
\frac{(m-1)(n-k)}{2(m+1)}>0.
\]

If the linear form representing $\xi_2$ has only one
nonzero coefficient, interchange the $\zeta$ coordinates
if necessary and write
\[
\xi_2=d\zeta_2,\qquad
\xi_1=b\zeta_1+e\zeta_2,\qquad bd\ne0.
\]
The weighted principal part is
\[
a(0,0)b^k\zeta_1^k
\bigl(d\zeta_2-cb^m\zeta_1^m\bigr)^n,
\]
which has Newton distance $d_k$.
Thus $h_{\mathrm{lin}}(\phi)=d_k$.

In the smooth coordinates
\[
u=\xi_1,\qquad v=\xi_2-\psi(\xi_1),
\]
the Taylor series has Newton polyhedron with vertex $(k,n)$
and Newton distance $n$. Hence
\[
h(\phi)\ge n>d_k,
\]
so the original coordinates are not adapted.
Also,
\[
\partial_{\xi_1}^k\partial_{\xi_2}^n\phi(0,0)
=k!n!a(0,0)\ne0.
\]
This derivative remains nonzero after the neighborhood is
reduced, ensuring finite type there.

The principal root has leading term $c\xi_1^m$.
Proposition~4.3 of
Palle~\cite{Palle} therefore supplies the exterior extension
estimate for
\[
q\ge2d_k+2,\qquad q<\infty.
\]
Using the same cutoff decomposition as in the main theorem,
it remains to consider data supported in the horn
\begin{equation}
	\label{eq:extension-horn}
	|\xi_2-c\xi_1^m|
	\le C\eta|\xi_1|^m,
\end{equation}
where $\eta>0$ is sufficiently small and fixed.
The set $\xi_1=0$ has parameter measure zero.

\subsection{Annular normalization}

Decompose the horn into dyadic annuli
$r\le|\xi_1|<2r$ and treat the two signs of $\xi_1$
separately. Define
\[
\xi_1=ru,\qquad \xi_2=r^my,\qquad
\phi_r(u,y)=r^{-\nu}\phi(ru,r^my).
\]
Then
\begin{equation}
	\label{eq:extension-normalized}
	\phi_r(u,y)
	=
	u^ka(ru,r^my)
	\bigl(y-\psi_r(u)\bigr)^n+R_r(u,y),
\end{equation}
where
\[
\psi_r(u)=r^{-m}\psi(ru),\qquad
R_r(u,y)=r^{-\nu}R(ru,r^my).
\]
The normalized domain lies in a fixed compact set with
$u\in[1,2]$ or $u\in[-2,-1]$.

For every fixed integer $N\ge0$,
\[
\psi_r\longrightarrow cu^m,\qquad
a(ru,r^my)\longrightarrow a(0,0)
\quad\text{in }C^N.
\]
Indeed,
\[
\psi_r(u)=u^m(c+ru\omega(ru)).
\]
For all integers $N,K\ge0$, flatness also gives
\begin{equation}
	\label{eq:extension-flat}
	\|R_r\|_{C^N}\le C_{N,K}r^K.
\end{equation}
For $i+j\le N$,
\[
\partial_u^i\partial_y^jR_r
=
r^{i+mj-\nu}
(\partial_1^i\partial_2^jR)(ru,r^my).
\]
Every derivative of $R$ is flat, so Taylor's theorem with
an arbitrarily large remainder order proves
\eqref{eq:extension-flat}.

Consequently,
\begin{equation}
	\label{eq:extension-limit}
	\phi_r\longrightarrow H
	\quad\text{in every fixed }C^N\text{ norm},
	\qquad
	H(u,y)=b(u)(y-B(u))^n,
\end{equation}
where
\[
b(u)=a(0,0)u^k,\qquad B(u)=cu^m.
\]
On each normalized interval,
\begin{equation}
	\label{eq:extension-nondegeneracy}
	|b(u)|\ge b_*>0,\qquad |B''(u)|\ge\kappa_*>0.
\end{equation}

\subsection{Uniform local estimates for $n\ge3$}

For the limiting phase,
\[
\partial_y^{n-1}H(u,y)=n!b(u)(y-B(u)),
\qquad
\partial_y^nH(u,y)=n!b(u).
\]
The quantitative implicit function theorem gives a smooth
curve $A_r$ satisfying
\[
\partial_y^{n-1}\phi_r(u,A_r(u))=0,
\qquad
A_r\longrightarrow B\quad\text{in }C^3.
\]
Thus $|A_r''|\ge\kappa_*/2$ for sufficiently small $r$.

On a fixed-width strip around $A_r$, the derivative
$\partial_y^n\phi_r$ has a fixed sign on each normalized
interval, and its absolute value is bounded above and below
uniformly. A negative sign is handled by reflecting the
third physical coordinate.
Proposition~\ref{prop:strip} yields
\begin{equation}
	\label{eq:extension-high}
	\|\mathcal E_{\phi_r}G\|_{L^q(\mathbb R^3)}
	\le C_q\|G\|_2,\qquad q>2n.
\end{equation}
The constant is independent of sufficiently small $r$.
After choosing $\eta$ and then $r$ small enough, the
normalized horn lies in the strip required by the
proposition.

\subsection{Uniform local estimates for $n=2$}

The limiting phase now has the form
\[
H(u,y)=b(u)(y-B(u))^2.
\]
Fix a point $u_0$ in either normalized interval and put
\[
\gamma=\tfrac12B''(u_0),\qquad b_0=b(u_0).
\]
Both quantities are nonzero.
For a small fixed $h>0$, use the affine ambient frequency
transformation
\begin{equation}
	\label{eq:extension-affine}
	\begin{aligned}
		u&=u_0+hs,\\
		y&=B(u_0)+B'(u_0)hs+\gamma h^2t,\\
		z&=b_0\gamma^2h^4\zeta.
	\end{aligned}
\end{equation}
The transformed graph phase is
\[
\widetilde\phi_r(s,t)
=
\frac{
	\phi_r\bigl(
	u_0+hs,\,
	B(u_0)+B'(u_0)hs+\gamma h^2t
	\bigr)}
{b_0\gamma^2h^4}.
\]
Its limiting phase is
\[
\widetilde H_h(s,t)
=
\frac{b(u_0+hs)}{b_0}
\left(
t-
\frac{
	B(u_0+hs)-B(u_0)-B'(u_0)hs
}{\gamma h^2}
\right)^2.
\]
Taylor's theorem gives, on every fixed enlarged box,
\[
\frac{
	B(u_0+hs)-B(u_0)-B'(u_0)hs
}{\gamma h^2}
=s^2+O_{C^4}(h),
\]
and
\[
\frac{b(u_0+hs)}{b_0}=1+O_{C^4}(h).
\]
Hence
\[
\|\widetilde H_h-(t-s^2)^2\|_{C^4}\le Ch.
\]

Let $\varepsilon_0$ be the threshold in
Proposition~\ref{prop:square}.
First fix $h$ so that this error is less than
$\varepsilon_0/2$. With $h$ fixed,
\eqref{eq:extension-limit} implies
\[
\|\widetilde\phi_r-\widetilde H_h\|_{C^4}
<\varepsilon_0/2
\]
for sufficiently small $r$.
All norms are taken on the enlarged box required by that
proposition.

For data in the normalized horn and $|s|\le1$,
\[
|t-s^2|\le C\eta h^{-2}+Ch.
\]
After fixing $h$, choose $\eta$ sufficiently small so that
the transformed data lie in the smaller domain required
by Proposition~\ref{prop:square}.

The parameter Jacobian in \eqref{eq:extension-affine} is
$|\gamma|h^3$, while the absolute determinant of the ambient
linear part is $|b_0||\gamma|^3h^7$.
The extension constant therefore changes by the fixed
factor
\[
(|\gamma|h^3)^{1/2}
\bigl(|b_0||\gamma|^3h^7\bigr)^{-1/q},
\]
which is independent of $r$.
A finite covering of the two normalized intervals gives
\begin{equation}
	\label{eq:extension-square}
	\|\mathcal E_{\phi_r}G\|_{L^q(\mathbb R^3)}
	\le C_q\|G\|_2,\qquad q>\frac{14}{3},
\end{equation}
uniformly for sufficiently small $r$.
This suffices because
\[
2+2d_k
=
2+\frac{2(k+2m)}{m+1}
\ge2+\frac{4m}{m+1}
\ge\frac{14}{3}.
\]

\subsection{Scaling back and summing the annuli}

Fix $Q<q<\infty$.
For data $g_r$ supported on one original annulus, set
\[
G_r(u,y)=g_r(ru,r^my).
\]
The parameter Jacobian is $r^{m+1}$.
The ambient dilation has determinant $r^{m+1+\nu}$.
Therefore
\[
\mathcal E_\phi g_r(x)
=
r^{m+1}
\mathcal E_{\phi_r}G_r(rx_1,r^mx_2,r^\nu x_3).
\]
By \eqref{eq:extension-high} or
\eqref{eq:extension-square},
\[
\|\mathcal E_\phi g_r\|_q
\le C_q r^{\beta(q)}\|g_r\|_2,
\qquad
\beta(q)=\frac{m+1}{2}-\frac{m+1+\nu}{q}.
\]
Since $q>2+2d_k$, one has $\beta(q)>0$.
For disjoint dyadic annuli,
\[
\begin{aligned}
	\left\|\sum_r\mathcal E_\phi g_r\right\|_q
	&\le C_q\sum_r r^{\beta(q)}\|g_r\|_2\\
	&\le C_q
	\left(\sum_r r^{2\beta(q)}\right)^{1/2}
	\left(\sum_r\|g_r\|_2^2\right)^{1/2}\\
	&\lesssim_q\|g\|_2.
\end{aligned}
\]
Combining this with the exterior estimate proves
sufficiency. The case $q=\infty$ follows directly from
Cauchy--Schwarz inequality on the bounded parameter domain.

\subsection{Necessary conditions}

Let $d\mu=w(\xi)\,d\xi$ be a graph measure with continuous
density positive at the origin.
Fix $M\ge m$ and use the parameters
\[
u=\xi_1,\qquad v=\xi_2-\psi(\xi_1).
\]
Consider
\[
B_{r,M}=\{|u|\le r,\ |v|\le r^M\}.
\]
The parameter Jacobian is one. On this box,
\[
|\xi_1|\le r,\qquad
|\xi_2|\le Cr^m,\qquad
|\xi_1^ka(\xi)v^n|\le Cr^{k+Mn}.
\]
For an integer $L>k+Mn$, flatness gives
\[
|R(u,v+\psi(u))|\le C_Lr^L.
\]
Thus $|\phi|\le Cr^{k+Mn}$ for sufficiently small $r$.

Let $f_{r,M}$ be the indicator of this box, transported
to the original parameters. Then
\[
\|f_{r,M}\|_{L^2(d\mu)}
\approx  r^{(M+1)/2}.
\]
On the physical box
\[
|x_1|\le\varepsilon r^{-1},\qquad
|x_2|\le\varepsilon r^{-m},\qquad
|x_3|\le\varepsilon r^{-(k+Mn)},
\]
the oscillation is small if $\varepsilon>0$ is fixed
sufficiently small.
Positivity of the density yields
\[
|\mathcal E_\mu f_{r,M}(x)|\gtrsim r^{M+1}.
\]
An $L^2(d\mu)\to L^q$ estimate therefore requires
\[
r^{M+1-(k+Mn+m+1)/q}
\lesssim r^{(M+1)/2}.
\]
Letting $r\downarrow0$ gives
\[
q\ge\frac{2(k+Mn+m+1)}{M+1}.
\]
Taking $M=m$ yields $q\ge2+2d_k$.
Letting $M\to\infty$ in this family of necessary inequalities
gives $q\ge2n$. Hence $q\ge q_k$ is necessary.
For each fixed $M$, the limit $r\downarrow0$ is taken first;
no uniformity of the flatness constants in $M$ is required.

\subsection{Proof of Corollary~\ref{cor:extension-defining-function}}

For Corollary~\ref{cor:extension-defining-function}, the
implicit function theorem gives the smooth zero curve
$\xi_2=\psi(\xi_1)$. The fundamental theorem of calculus
gives the exact factorization
\[
F(\xi_1,\xi_2)
=
(\xi_2-\psi(\xi_1))B_F(\xi_1,\xi_2),
\]
where
\[
B_F(\xi_1,\xi_2)
=
\int_0^1
\partial_{\xi_2}F
\bigl(
\xi_1,\,
\psi(\xi_1)+t(\xi_2-\psi(\xi_1))
\bigr)\,dt.
\]
The coefficient is smooth and satisfies
\[
B_F(0,0)=\partial_{\xi_2}F(0,0)\ne0.
\]
Consequently,
\[
\xi_1^ka(\xi)F(\xi)^n+R(\xi)
=
\xi_1^k\bigl(a(\xi)B_F(\xi)^n\bigr)
(\xi_2-\psi(\xi_1))^n+R(\xi).
\]
The new coefficient is nonzero at the origin, so
Theorem~\ref{thm:extension-tangential} applies.

\begin{remark}
	The substitutions $v=\xi_2-\psi(\xi_1)$ above only
	reparametrize surface integrals or describe test sets.
	They are not used to transfer restriction estimates through
	nonlinear ambient transformations.
	Every ambient transformation used in the normalization
	of an extension estimate is affine, and its Jacobian has
	been explicitly accounted for.
\end{remark}

\par\bigskip
\begingroup
\small
\raggedright

\textbf{Wenjuan Li}\\
School of Mathematics and Statistics, Northwestern Polytechnical University, Xi'an, 710129, China\\
\textit{E-mail:} \texttt{liwj@nwpu.edu.cn}
\par\medskip

\textbf{Hongqiang Wang}\\
School of Mathematics and Statistics, Northwestern Polytechnical University, Xi'an, 710129, China\\
\textit{E-mail:} \texttt{whq2023@mail.nwpu.edu.cn}
\par\medskip

\textbf{Huiju Wang}\\
School of Mathematics and Statistics, Henan University, Kaifeng, 475000, China\\
\textit{E-mail:} \texttt{huijuwang@mail.nwpu.edu.cn}
\par

\endgroup

\end{document}